%% file: boundarycc.tex
\documentclass[letter,11pt]{amsart}
\usepackage{amsmath, amssymb, amsthm, latexsym}
\usepackage[alphabetic]{amsrefs}
\usepackage{xspace}
\usepackage{fullpage}
\usepackage[usenames,dvipsnames,svgnames]{xcolor}
\usepackage{tikz}
\usepackage{chngcntr}
\usepackage{hyperref}

\counterwithin{figure}{section}

\newtheorem{proposition}{Proposition}[section]
\newtheorem{theorem}[proposition]{Theorem}
\newtheorem{corollary}[proposition]{Corollary}

\newtheorem{lemma}[proposition]{Lemma}

\theoremstyle{definition}
\newtheorem{definition}[proposition]{Definition}
\newtheorem{question}[proposition]{Question}
\newtheorem{remark}[proposition]{Remark}

\newcommand{\set}[1]{\left\{#1\right\}}
\newcommand{\setcon}[2]{\left\{#1\ \left|\ #2\right.\right\}}

\newcommand{\C}{\mathbb{C}}
\newcommand{\R}{\mathbb{R}}
\newcommand{\Z}{\mathbb{Z}}
\newcommand{\N}{\mathbb{N}}
\newcommand{\PSL}{\mathrm{PSL}}
\newcommand{\Out}{\mathrm{Out}(\mathbb F)}

\newcommand{\fgen}[1]{\left\langle #1 \right\rangle}

\newcommand{\Mod}{\mathrm{Mod}(S)}
\newcommand{\teich}{\mathcal{T}(S)}

\author[M. Cordes]{Matthew Cordes}
\address{Department of Mathematics\\
Brandeis University\\
415 South Street\\
Waltham, MA 02453, U.S.A.}
\email{\href{mailto:mcordes@brandeis.edu}{mcordes@brandeis.edu}}
\author[M.G. Durham]{Matthew Gentry Durham}
\address{Department of Mathematics\\ 
University of Michigan\\ 
530 Church Street\\ 
Ann Arbor, MI 48105, U.S.A.}
\email{\href{mailto:durhamma@umich.edu}{durhamma@umich.edu}}
\date{\today}

\title{Boundary convex cocompactness and stability of subgroups of finitely generated groups}

\begin{document}

\setcounter{tocdepth}{1}
\maketitle

\begin{abstract}
A Kleinian group $\Gamma < \mathrm{Isom}(\mathbb H^3)$ is called convex cocompact if any orbit of $\Gamma$ in $\mathbb H^3$ is quasiconvex  or, equivalently, $\Gamma$ acts cocompactly on the convex hull of its limit set in $\partial \mathbb H^3$.

Subgroup stability is a strong quasiconvexity condition in finitely generated groups which is intrinsic to the geometry of the ambient group and generalizes  the classical quasiconvexity condition above.  Importantly, it coincides with quasiconvexity in hyperbolic groups and convex cocompactness in mapping class groups.

Using the Morse boundary, we develop an equivalent characterization of subgroup stability which generalizes the above boundary characterization from Kleinian groups. 
\end{abstract}


\section{Introduction}

There has been much recent interest in generalizing salient features of Gromov hyperbolic spaces to more general contexts, including their boundaries and convexity properties of nicely embedded subspaces.  Among these are the Morse property, its generalization to subspaces, stability, and the Morse boundary.  In this article, we use the Morse boundary to prove that stability for subgroups of finitely generated groups is naturally a convex cocompactness condition in the classical boundary sense.  We begin with some motivation from Kleinian groups and mapping class groups.

A nonelementary discrete (Kleinian) subgroup $\Gamma < \PSL_2(\C)$ determines a minimal $\Gamma$-invariant closed subspace $\Lambda(\Gamma)$ of the Riemann sphere called its \emph{limit set} and taking the convex hull of $\Lambda(\Gamma)$ determines a convex subspace of $\mathbb H^3$ with a $\Gamma$-action.  A Kleinian group $\Gamma$ is called \emph{convex cocompact} if it acts cocompactly on this convex hull or, equivalently, any $\Gamma$-orbit in $\mathbb H^3$ is quasiconvex.  Another equivalent characterization is that such a $\Gamma$ has a compact Kleinian manifold; see \cite{marden1974geometry, sullivan1985quasiconformal}.

Originally defined by Farb-Mosher \cite{FarbMosher} and later developed further by Kent-Leininger \cite{KentLein} and Hamenst\"adt \cite{HamCC}, a subgroup $H < \Mod$ is called convex cocompact if and only if any $H$-orbit in $\teich$, the Teichm\"uller space of $S$ with the Teichm\"uller metric, is quasiconvex, or $H$ acts cocompactly on the weak hull of its limit set $\Lambda(H) \subset \mathbb{P}\mathcal{ML}(S)$ in the Thurston compactification of $\teich$.  This notion is important because convex cocompact subgroups $H< \Mod$ are precisely those which determine Gromov hyperbolic surface group extensions.

In both of these examples, convex cocompactness is characterized equivalently by both a quasiconvexity condition and an asymptotic boundary condition.  In \cite{DT15}, Taylor and the second author introduced stability in order to characterize convex cocompactness in $\Mod$ by a quasiconvexity condition intrinsic to the geometry of $\Mod$.  In fact, stability naturally generalizes the above quasiconvexity characterizations of convex cocompactness to any finitely generated group.

In this article, we use the Morse boundary to define an asymptotic property for subgroups of finitely generated groups called \emph{boundary convex cocompactness} which generalizes the classical boundary characterization of convex cocompactness from Kleinian groups.  Our main theorem is:  

\begin{theorem}\label{thm:main}
Let $G$ be a finitely generated group.  Then $H<G$ is boundary convex cocompact if and only if $H$ is stable in $G$.
\end{theorem}

Before moving on to the definitions, we discuss the situation in hyperbolic groups.

Let $H$ be a quasiconvex subgroup of a hyperbolic group $G$.   Then $H$ has a limit set $\Lambda(H) \subset \partial_{Gr} G$, the Gromov boundary of $G$, and one can define the weak hull of $\Lambda(H)$ to be the union of all geodesics in $G$ connecting distinct points in $\Lambda(H)$.  Swenson \cite{swenson2001quasi} proved that $H$ acts cocompactly on this weak hull if and only $H$ is quasiconvex in $G$.  Hence a quasiconvex subgroup of a hyperbolic group satisfies a boundary characterization of convex cocompactness intrinsic to the ambient geometry.

Stability generalizes quasiconvexity to any finitely generated group, and similarly the Morse boundary generalizes the Gromov boundary.  Thus Theorem \ref{thm:main} is a generalization of the hyperbolic case to any finitely generated group.

\subsection*{Stability and the Morse boundary}

We start with the definition of a Morse quasigeodesic:

\begin{definition}[Morse quasigeodesic]
A quasigeodesic $\gamma$ in geodesic metric space $X$ is called $N$-\emph{Morse} if there exists a function $N:\R_{\geq 1}\times \R_{\geq 0} \rightarrow \R_{\geq 0}$ such that if $q$ is any $(K,C)$-quasigeodesic with endpoints on $\gamma$, then $q \subset \mathcal N_{N(K,C)}(\gamma)$, the $N(K,C)$-neighborhood of $\gamma$.

We call $N$ the \emph{Morse gauge} of $\gamma$.
\end{definition}

Note that if $X$ is hyperbolic, then every geodesic in $X$ is Morse with a uniform gauge.

If $G$ is a finitely generated group, then we call $g \in G$ a \emph{Morse element} if its orbit in any Cayley graph of $G$ is a Morse quasigeodesic.  Some examples of Morse elements include rank-1 elements of CAT(0) groups \cite{BehrstockCharney}, pseudo-Anosov elements of mapping class groups \cite{Behrstock:asymptotic}, fully irreducible elements of the outer automorphism groups of free groups \cite{algom2011strongly}, and rank-one automorphisms of hierarchically hyperbolic spaces \cite{durham2016boundaries}.

We can now give a formal definition of stability:

\begin{definition}[Stability]
If $f:X \rightarrow Y$ is a quasiisometric embedding between geodesic metric spaces, we say $X$ is a \emph{stable} subspace of $Y$ if there exists a Morse gauge $N$ such that every pair of points in $X$ can be connected by an $N$-Morse quasigeodesic in $Y$; we call $f$ a \emph{stable embedding}.

If $H<G$ are finitely generated groups, we say $H$ is \emph{stable in} $G$ if the inclusion map $i:H \hookrightarrow G$ is a stable embedding.
\end{definition}

We note that stable subgroups are always hyperbolic and quasiconvex regardless of the chosen word metric on $G$---stability and quasiconvexity coincide in hyperbolic spaces---and stability is invariant under quasiisometric embeddings \cite{DT15}.

Introduced by the first author in \cite{Cordes15} for proper geodesic spaces and later refined and generalized to geodesic spaces in \cite{cordes2016stability}, the \emph{Morse boundary} of a geodesic metric space $X$, denoted $\partial_M X_e$, consists, roughly speaking, of asymptotic classes of sequences of points which can be connected to a fixed basepoint $e \in X$ by Morse geodesic rays; see Subsection \ref{subsec:morse boundary def} for the formal definition.  Importantly, it is a visual boundary, generalizes the Gromov boundary when $X$ is hyperbolic, and quasiisometries induce homeomorphisms at the level of Morse boundaries.

\subsection*{Boundary convex cocompactness}

Let $G$ be a finitely generated group acting by isometries on a proper geodesic metric space $X$.  Fix a basepoint $e \in X$.  One can define a limit set $\Lambda_e(G) \subset \partial_M X_e$ as the set of points which can be represented by sequences of $G$-orbit points;  note that $\Lambda_e(G)$ is obviously $G$-invariant.  One then defines the weak hull $\mathfrak H_e(G)$ of $\Lambda_e(G)$ in $X$ by taking all geodesics with distinct endpoints in $\Lambda_e(G)$.  See Section \ref{sec:G action} for the precise definitions.

\begin{definition}[Boundary convex cocompactness] \label{defn:cc subspace}
We say that $G$ acts \emph{boundary convex cocompactly} on $X$ if the following conditions hold:
\begin{enumerate}
\item $G$ acts properly on $X$;
\item For some (any) $e \in X$, $\Lambda_e(G)$ is nonempty and compact; 
\item For some (any) $e \in X$, the action of $G$ on $\mathfrak H_e(G)$ is cocompact. 
\end{enumerate} 
\end{definition}

\begin{definition}[Boundary convex cocompactness for subgroups] \label{defn:cc subgroup}
Let $G$ be a finitely generated group.  We say $H<G$ is \emph{boundary convex cocompact} if $H$ acts boundary convex cocompactly on any Cayley graph of $G$ with respect to a finite generating set. 
\end{definition}

Theorem \ref{thm:main} is an immediate consequence of the following stronger statement:

\begin{theorem} \label{thm:cc and stab}
Let $G$ be a finitely generated group acting by isometries on a proper geodesic metric space $X$.  Then the action of $G$ is boundary convex cocompact if and only if some (any) orbit of $G$ in $X$ is a stable embedding.

In either case, $G$ is hyperbolic and any orbit map $orb_e:G \rightarrow X$ extends continuously and $G$-equivariantly to an embedding of $\partial_{Gr} G$ which is a homeomorphism onto its image $\Lambda_e(G) \subset \partial_M X_e$.
\end{theorem}

We note that Theorem \ref{thm:main} and \cite[Proposition 3.2]{DT15} imply that boundary convex cocompactness is invariant under quasiisometric embeddings.

\begin{remark}[On the necessity of the conditions in Definition \ref{defn:cc subspace}]
The compactness assumption on $\Lambda_e(G)$ is essential: Consider the group $G=\Z^2 * \Z * \Z = \fgen{a,b} * \fgen{c} * \fgen{d}$ acting on its Cayley graph with the subgroup $H=\fgen{a,b,c}$.  Since the $H$ is isometrically embedded and convex in $G$, it follows that $\partial_M H_e\cong \Lambda_e(H) \subset \partial_M G_e$ and $\mathfrak H_e(H) = H$ for any $e \in G$, whereas $H$ is not hyperbolic and thus not stable in $G$.

While compactness of $\Lambda_e(G)$ does imply that $\mathfrak H_e(G)$ is stable in $X$ (Proposition \ref{prop:compact hyp hull}), it is unclear how to leverage this fact into proving properness or cocompactness of the $G$-action, even in the presence of one or the other. 
\end{remark}

\begin{remark}[Stability versus boundary convex cocompactness]
We expect that stability will be a much condition to check in practice than boundary convex cocompactness.  Our purpose is to prove that stability generalizes multiple classical notions of convex cocompactness and is thus the correct generalization of convex cocompactness to finitely generated groups.
\end{remark}

\begin{remark}[Conical limit points]
In addition to the above discussed notions for Kleinian groups, there is characterization for convex cocompactness in terms of limit points, namely that every limit point is conical.  Kent-Leininger develop a similar characterization of this for subgroups of $\Mod$ acting on $\teich \cup \mathbb{P}\mathcal{ML}(S)$.  We believe that there is a similar conicality characterization for subgroup stability with respect to limit sets in the Morse boundary.  However, we have chosen not to pursue this here in interest of brevity.  Moreover, it's unclear how useful such a characterization would be.
\end{remark}

\subsection*{Stability and boundary convex cocompactness in important examples}

Since stability is a strong property, it is interesting to characterize and produce stable subgroups of important groups.  There has been much recent work to do so, which we will now briefly overview.

\begin{enumerate}
\item For a relatively hyperbolic group $(G, \mathcal P)$, Aougab, Taylor, and the second author \cite{ADT16} prove that if $H < G$ is finitely generated and quasiisometrically embeds in the associated coned space \cite{farb1998relatively}, then $H$ is stable in $G$.  Moreover, if we further assume that the subgroups in $P$ are one-ended and have linear divergence, then quasiisometrically embedding in the coned space is equivalent to stability.
\item For $\Mod$, that subgroup stability and convex cocompactness in the sense of \cite{FarbMosher} are equivalent was proven by Taylor and the second author \cite{DT15}.
\item For the right-angled Artin group $A(\Gamma)$ of a finite simplicial graph $\Gamma$ which is not a join, Koberda-Mangahas-Taylor \cite{koberda2014geometry} proved that stability for $H<A(\Gamma)$ is equivalent to $H$ being finitely generated and purely loxodromic, i.e. each element acts loxodromically on the associated extension graph $\Gamma^e$, a curve graph analogue.  They also prove that stable subgroups satisfy a strictly weaker condition called combinatorial quasiconvexity, sometimes called convex cocompactness \cite{Haglund08}; we note that, unlike stable subgroups, combinatorially quasiconvex subgroups need not be hyperbolic.
\item For $\Out$, the outer automorphism group of a free group on at least three letters, work in \cite{ADT16} proves that a subgroup $H < \Out$ which has a quasiisometrically embedded orbit in the free factor graph $\mathcal F$ is stable in $\Out$.  This is related to and builds on others' work as follows: Hamenst\"adt-Hensel \cite{hamenstadt2014convex} proved that a subgroup $H < \Out$ quasiisometrically embedding in the free factor graph is equivalent to a certain convex cocompactness condition in the projectivization of the Culler-Vogtmann outer space and its boundary which is analogous to that the Kent-Leininger condition on $\teich \cup \mathbb{P}\mathcal{ML}(S)$.  Following \cite{hamenstadt2014convex}, we shall call such subgroups convex cocompact.  By work of Dowdall-Taylor \cite{dowdall2014hyperbolic, dowdall2015contracting}, quasiisometrically embedding in the free factor graph implies a stability-like property for orbit maps in outer space; we note that when such a group is also fully atoroidal, they prove the corresponding free group extension is hyperbolic.  In \cite{ADT16}, the authors prove this stability property pulls back to genuine stability in $\Out$.
\end{enumerate}
 
We summarize these results in the following theorem:

\begin{theorem}[\cite{DT15, koberda2014geometry, ADT16}]\label{thm:stability in examples}
Suppose that the pair $H<G$ satisfies one of the following conditions:

\begin{enumerate}
\item $H$ is a quasiconvex subgroup of a hyperbolic group $G$;
\item $G$ is relatively hyperbolic and $H$ is finitely generated and quasiisometrically embeds in the coned space associated to $G$ in the sense of \cite{farb1998relatively};
\item $G = A(\Gamma)$ for $\Gamma$ a finite simplicial graph which is not a join and $H$ is finitely generated and $H$ quasiisometrically embeds in the extension graph $\Gamma^e$ \cite{koberda2014geometry};
\item $G = \Mod$ and $H$ is a convex cocompact subgroup in the sense of \cite{FarbMosher};
\item $G = \Out$ and $H$ is a convex cocompact subgroup in the sense of \cite{hamenstadt2014convex}.
\end{enumerate}
Then $H$ is stable in $G$.  Moreover, for (1), (3), and (4), the reverse implication also holds.
\end{theorem}

As a corollary of Theorem \ref{thm:main}, we have:

\begin{corollary}\label{cor:cc ex}
Suppose $H<G$ are as any of (1)--(5) in Theorem \ref{thm:stability in examples}.  Then $H$ is boundary convex cocompact.
\end{corollary}

Item (1) in Corollary \ref{cor:cc ex} is originally due to Swenson \cite{swenson2001quasi}.  In examples (2)--(5), each of the previously established boundary characterizations of convex cocompactness was in terms of an external space.  Corollary \ref{cor:cc ex} provides a boundary characterization of convex cocompactness which is intrinsic to the geometry of the ambient group.
  
We note that Hamenst\"adt has announced that there are stable subgroups of $\Out$ which are not convex cocompact in the sense of \cite{hamenstadt2014convex}, but such subgroups would be convex cocompact both in the quasiconvex and boundary senses.

Finally, we see that boundary convex cocompactness is generic in these main example groups.  For any such group $G$ and probability measure $\mu$ thereon, consider $k\geq 2$ independent random walks $(w^1_n)_{n \in \mathbb N}, \dots, (w^k_n)_{n \in \mathbb N}$ whose increments are distributed according to $\mu$.  For each $n$, let $\Gamma(n) = \langle w^1_n, \dots, w^k_n\rangle \leq G$.  Following Taylor-Tiozzo \cite{taylor2016random}, we say a random subgroup of $G$ has a property $P$ if 
$$\mathbb P[\Gamma(n) \text{ has } P] \rightarrow 1.$$

In the following, (1) is a consequence of \cite{taylor2016random}, (3) of \cite{koberda2014geometry}, and (2), (4)--(5) of \cite{ADT16}:

\begin{theorem}[\cite{taylor2016random, koberda2014geometry, ADT16}]\label{thm:stab random}
Suppose that $G$ is any of the groups in (1)--(5) of Theorem \ref{thm:stability in examples}.  Then a $k$-generated random subgroup of $G$ is stable.
\end{theorem}

Hence, Theorem \ref{thm:main} gives us:

\begin{corollary}\label{cor:cc random}
Suppose $G$ is any of the groups in (1)--(5) of Theorem \ref{thm:stability in examples}.  Then a $k$-generated random subgroup of $G$ is boundary  convex cocompact.
\end{corollary}

\subsection*{Acknowledgements}

We would like to thank the organizers of the ``Geometry of Groups in Montevideo" conference, where part of this work was accomplished.  The first and second authors were partially supported by NSF grants DMS-1106726 and DMS-1045119, respectively.  We would like to thank Ursula Hamenst\"adt and Samuel Taylor for interesting conversations, and also the latter for useful comments on an earlier draft of this paper. 
 
\section{Background}
We assume the reader is familiar with basics of $\delta$-hyperbolic spaces and their Gromov boundaries. In this subsection we recall some of the basic definitions. For more information, see \cite[III.H]{BridsonHaefliger:book}.

\subsection{Gromov boundaries}

\begin{definition}
Let $X$ be a metric space and let $x,y,z\in X$. The \emph{Gromov product} of $x$ and $y$ with respect to $z$ is defined as $$ (x\cdot y)_z = \frac{1}{2}\left( d(z,x)+d(z,y)-d(x,y)\right).$$
Let $(x_n)$ be a sequence in $X$. We say $(x_i)$ converges at infinity if $(x_i\cdot x_j)_e \to \infty$ as $i,j \to \infty$. Two convergent sequences $(x_n),(y_m)$ are said to be \textbf{equivalent} if $(x_i \cdot y_j) \to \infty$ as $i,j \to \infty$. We denote the equivalence class of $(x_n)$ by $\lim x_n$.

The \textbf{sequential boundary} of $X$, denoted $\partial X$, is defined to be the set of convergent sequences considered up to equivalence.
\end{definition}

\begin{definition}\label{defn:hyp} \cite[Definition $1.20$]{BridsonHaefliger:book} Let $X$ be a (not necessarily geodesic) metric space. We say $X$ is $\delta$\emph{--hyperbolic} if for all $w,x,y,z$ we have
$$ (x\cdot y)_w \geq \min\set{(x\cdot z)_w,(z\cdot y)_w} - \delta.$$
\end{definition}

If  $X$ is $\delta$--hyperbolic, we may extend the Gromov product to $\partial  X$ in the following way:
$$ (x\cdot y)_e = \sup \left(\liminf_{m,n\to\infty}\set{(x_n\cdot y_m)_e}\right).$$
where $x,y \in \partial X$ and the supremum is taken over all sequences $(x_i)$ and $(y_j)$ in $X$ such that $x= \lim x_i$ and $y=\lim y_j$.

\subsection{(Metric) Morse boundary} \label{subsec:morse boundary def}
Let $X$ be a geodesic metric space, let $e\in X$ and let $N$ be a Morse gauge. We define $X^{(N)}_e$ to be the set of all $y\in X$ such that there exists a $N$--Morse geodesic $[e,y]$ in $X$.

\begin{proposition}[$X^{(N)}_e$ are hyperbolic; Proposition 3.2 \cite{cordes2016stability}]  \label{prop:hypsubsets} $X^{(N)}_e$ is $8N(3,0)$--hyperbolic in the sense of Definition $\ref{defn:hyp}$. \end{proposition}

As each $X^{(N)}_e$ is hyperbolic we may consider its Gromov boundary, $\partial X^{(N)}_e$, and the associated visual metric $d_{(N)}$. We call the collection of boundaries $\left( \partial X^{(N)}_e, d_{(N)} \right)$ the \emph{metric Morse boundary} of $X$. 

Instead of focusing on sequences which live in some $X^{(N)}_e$, we now consider the set of all Morse geodesic rays in $X$ (with basepoint $p$) up to asymptotic equivalence.  We call this collection the  \emph{Morse boundary} of $X$, and denote it by $\partial_M X$.

To topologize the boundary, first fix a Morse gauge $N$ and consider the subset of the Morse boundary that consists of all rays in $X$ with Morse gauge at most $N$:  \begin{equation*} \partial_M^N X_p= \{[\alpha] \mid \exists \beta \in [\alpha] \text{ that is an $N$--Morse geodesic ray with } \beta(0)=p\}. \end{equation*} We topologize this set with the compact-open topology. This topology is equivalent to one defined by a system of neighborhoods, $\{V_n(\alpha) \mid n \in \N \}$, at a point $\alpha$ in $\partial_M^N X_p$, which are defined as follows: the set $V_n( \alpha)$ is the set of geodesic rays $\gamma$ with basepoint $p$ and $d(\alpha(t), \gamma(t))< \delta_N$ for all $t<n$, where $\delta_N$ is a constant that depends only on $N$. 

Let $\mathcal M$ be the set of all Morse gauges. We put a partial ordering on $\mathcal M$ so that  for two Morse gauges $N, N' \in \mathcal M$, we say $N \leq N'$ if and only if $N(\lambda,\epsilon) \leq N'(\lambda,\epsilon)$ for all $\lambda,\epsilon \in \N$. We define the Morse boundary of $X$ to be
 \begin{equation*} \partial_M X_p=\varinjlim_\mathcal{M} \partial^N_M X_p \end{equation*} with the induced direct limit topology, i.e., a set $U$ is open in $\partial_M X_p$ if and only if $U \cap \partial^N_M X_p$ is open for all $N$.   

In \cite{cordes2016stability} the first author and Hume show that if $X$ is a proper geodesic metric space, then there is a natural homeomorphism between $\partial X^{(N)}_e$ and $\partial_M^N X_e$. 


\subsection{Useful facts}
In this subsection, we will collect a number of basic facts and definitions.

The following lemma states that a quasigeodesic with endpoints on a Morse geodesic stays Hausdorff close:

\begin{lemma}[Hausdorff close; Lemma 2.1 in \cite{Cordes15}] \label{lem:hausdorff} Let $X$ be a geodesic space and let $\gamma \colon [a, b] \to X$ be a $N$-Morse geodesic segment and let $\sigma \colon [a', b'] \to X$ be a continuous $(K,C)$-quasi-geodesic such that $\gamma(a)=\sigma(a')$ and $\gamma(b)=\sigma(b')$. Then the Hausdorff distance between $\alpha$ and $\beta$ is bounded by $2N(\lambda, \epsilon)$
\end{lemma}

The next lemma states that if a geodesic triangles $\Delta$ has two Morse sides, then $\Delta$ is slim and its third side is also Morse:

\begin{lemma}[Lemma 2.2-2.3 in \cite{Cordes15}]\label{lem:deltaslim} Let $X$ be a geodesic space.  For any Morse gauge $N$, there exists a gauge $N'$ such that the following holds:  Let  $\gamma_1, \gamma_2 \colon [0, \infty) \to X$ be $N$-Morse geodesics such that $\gamma_1(0)=\gamma_2(0)=e$ and let $x_1=\gamma_1(v_1),x_2=\gamma_2(v_2)$ be points on $\gamma_1$ and $\gamma_2$ respectively. Let $\gamma$ be a geodesic between $x_1$ and $x_2$. Then the geodesic triangle $\gamma_1([0,v_1]) \cup \gamma \cup \gamma_2([0, v_2])$ is $4N(3,0)$-slim.  Moreover $\gamma$ is $N'$-Morse.
\end{lemma}

The following is the standard fact that quasiisometries preserve the Gromov product:

\begin{lemma}\label{lem:qi preserves product}
Suppose $X$ and $Y$ are proper Gromov hyperbolic metric spaces.  If $f:X \rightarrow Y$ is a quasiisometric embedding, then there exist $A\geq 1, B>0$ such that for any $x,y,z \in X$, we have
$$\frac{1}{A} (x \cdot y)_z -B \leq (f(x) \cdot f(y))_{f(z)} \leq A (x \cdot y)_z + B.$$ 
\end{lemma}

We have the following well-known consequence:

\begin{lemma} \label{lem:qi hyp boundary}
Suppose that $f:X \rightarrow Y$ is a quasiisometric embedding between proper Gromov hyperbolic spaces.  Then the induced map $\partial f: \partial X \rightarrow \partial Y$ is a homeomorphism onto its image.
\end{lemma}

%
%
%

\subsection{Stability}

The following is \cite[Definition 3.1]{DT15}:
\begin{definition}[Stability 1]
if for any $K \geq 1, C \geq 0$, there exists $R=R(K,C)\geq 0$ such that if $q_1:[a,b]\rightarrow Y,  q_2:[a',b']\rightarrow Y$ are $(K,C)$-quasigeodesics with $q_1(a) = q'_1(a'), q_1(b) = q_2(b') \in f(X)$, then 
$$d_{Haus}(q_1, q_2) < R.$$
\end{definition}

We will use the following definition of stability which is equivalent via Lemma \ref{lem:hausdorff}:

\begin{definition}[Stability 2]\label{def:stability}
Let $X, Y$ be geodesic metric spaces and let $f: X \rightarrow Y$ be a quasiisometric embedding.  We say $X$ is \emph{stable} in $Y$ there exists a Morse gauge $N$ such that any $x,y \in f(X)$ are connected by a $N$-Morse geodesic in $Y$.  We say that $f$ is a \emph{stable embedding}.
\end{definition}

\subsection{Morse preserving maps}

It will be useful for having a notion of when maps between metric spaces preserve data encoded by the Morse boundary.

\begin{definition}[Morse preserving maps]
Let $X, Y$ be proper geodesic metric spaces and $e \in X, e' \in Y$.  We say that $g:\partial_M X_e \rightarrow \partial_M Y_{e'}$ \emph{Morse-preserving} if for each Morse gauge $N$, there exists another Morse gauge $N'$ such that $g$ injectively maps $\partial X^{(N)}_e \rightarrow \partial Y^{(N')}_{e'}$.
\end{definition}

The following is a consequence of the definitions: 

\begin{proposition}\label{prop:morse preserving}
Let $f:X \rightarrow Y$ be $(K,C)$-quasiisometric embedding.  If $f$ induces a Morse-preserving map $\partial_M f: \partial_M X \rightarrow \partial_M Y$, then $\partial_M f$ is a homeomorphism onto its image.
\end{proposition}

We note that if $f:X \rightarrow Y$ is a stable embedding, then all geodesics get sent to uniformly Morse quasigeodesics.  Hence the induced map $\partial f:\partial_M X \rightarrow \partial_M Y$ is clearly Morse-preserving and:

\begin{corollary}\label{cor:stable morse preserving}
Let $X,Y$ be proper geodesic metric spaces.  If $f:X \rightarrow Y$ is a stable embedding, then the induced map $\partial f:\partial_M X \rightarrow \partial_M Y$ is a homeomorphism onto its image.  Moreover, if $f$ is the orbit map of a finitely generated group $G$ acting by isometries on $Y$, then $\partial f$ is $G$-equivariant.
\end{corollary}

\section{The action of $G$ on $\partial_M X$} \label{sec:G action}

In this section, we begin a study of the dynamics of the action of $G$ on $\partial_M X$.

\subsection{Definition of the $G$-action}

For the rest of this section, fix $G$ a group acting by isometries on a proper geodesic metric space $X$ and a base point $e \in X$.

\begin{lemma} \label{lem:G-action well-defined}
Given any Morse function $N$ and $g \in G$, there exists a Morse function $N'$ depending only on $N$ and $g$ such that if $(x_n), (y_n) \subset X^{(N)}_e$ are asymptotic, then:
\begin{enumerate}
\item $(g \cdot x_n), (g \cdot y_n) \subset X^{(N')}_e$ and
\item $(g \cdot y_n) \in X^{(N')}_e$ is asymptotic to $(g \cdot x_n) \in X^{(N')}_e$.
\end{enumerate}
\end{lemma}

\begin{proof}
Since $G$ is acting by isometries we see that $g$ and $g\cdot x_i$ are connected by a $N$-Morse geodesic. As in the proof of {Proposition 3.15} in \cite{cordes2016stability}  we show that $X^{(N)}_{g\cdot e} \subset X^{(N')}_e$ for some $N'$ which depends only on $N$ and $d(g\cdot e, e)$. 

To prove the second part, we note that in the proof of {Proposition 3.15} in \cite{cordes2016stability}  we also get that there exists a constant $D$ such that for all $x,y \in \partial X^{(N)}_{g\cdot e}$, $$(x \cdot_{N'} y)_e -D \leq (x \cdot_N y)_{g \cdot e} \leq (x \cdot_{N'} y)_e +D.$$

Since $G$ acts by isometries and $(x_n), (y_n) \subset X^{(N)}_e$ are asymptotic, it follows that $(g\cdot x_n), ( g \cdot y_n) \subset X^{(N)}_{g\cdot e}$ are asymptotic, which forces $(g\cdot x_n), ( g \cdot y_n) \subset X^{(N)}_{e}$ to be asymptotic. \end{proof}

We may naturally extend the action of $G$ on $X$ to an action on the whole Morse boundary, $\partial_M X$, as follows:  Suppose $(x_n) \subset X^{(N)}_e$ is a sequence $N$-converging to a point $\lambda \in \partial X^{(N)}_e$.  For any $g \in G$, we define $g \cdot \lambda$ to be the asymptotic class of $(g \cdot x_n) \in \partial X^{(N')}_e$, where $N'$ is the Morse gauge from Lemma \ref{lem:G-action well-defined}.

\subsection{Definition of the limit set}
    
\begin{definition}[$\Lambda(G)$] \label{defn:limit set}
The \emph{limit set} of the $G$-action on $\partial_M X$ is
$$\Lambda_e(G)= \setcon{\lambda \in \partial_M X}{\exists N \text{ and } (g_k) \subset G  \text{ such that } (g_k \cdot e)\subset X^{(N)}_e \text{ and } \lim g_k\cdot e=\lambda}$$
\end{definition}

The following lemma is an immediate consequence of the definitions and Lemma \ref{lem:G-action well-defined}:

\begin{lemma} 
For any $e, f \in X$ and corresponding natural change of basepoint homeomorphism $\phi_{e,f}: \Lambda_e(G) \rightarrow \Lambda_f(G)$, we have $\phi_{e,f}(\Lambda_e(G)) = \Lambda_f(G)$.  Moreover, we have $G \cdot \Lambda_e(G) \subset \Lambda_e(G)$, i.e., $\Lambda_e(G)$ is $G$-invariant. 
\end{lemma}

\subsection{Limit geodesics and their properties}

In Subsection \ref{subsec:weak hull}, we will define the weak hull of $\Lambda_e(G)$ as all geodesics connecting points in $\Lambda_e(G)$.  For now, however, we will focus our attention on a special class of visual geodesics arising naturally as limits of finite geodesics whose endpoint sequences converge to the Morse boundary.  These geodesics are easier to work with and more obviously tied to the geometry of $G$.



Let $(x_n), (y_n) \in X^{(N)}_e$ be sequences asymptotic to $\lambda^-, \lambda^+ \in \partial X^{(N)}_e$ respectively, with $\lambda^- \neq \lambda^+$.  Let $\gamma_{x,n} = [e, x_n]$, $\gamma_{y,n} = [e,y_n]$ be $N$-Morse geodesics.  Since $X$ is proper, the Arzel\`a-Ascoli theorem implies that there exist geodesic rays $\gamma_x, \gamma_y$ such that $\gamma_{x,n}, \gamma_{y,n}$ subsequentially converge uniformly on compact sets to $\gamma_x, \gamma_y$. Let $\gamma_n$ be a geodesic joining $\gamma_x(n)$ and $\gamma_y(n)$ for each $n \in \mathbb{N}$.

In fact, one can prove that $(\gamma_n)$ has a Morse subsequential limit:

\begin{lemma}\label{lem:limit geodesics are Morse}
Using the above notation, there exists a Morse gauge $N'$ and a biinfinite $N'$-Morse geodesic $\gamma: (-\infty, \infty) \rightarrow X$ such that $\gamma_n$ subsequentially converges uniformly on compact sets to $\gamma$.   Moreover, we have that $\gamma \subset \mathcal N_{4N(3,0)}(\gamma_x \cup \gamma_y)$.
\end{lemma}

\begin{proof}
Let $R \in \mathbb{N}$ be so that $d(\gamma_x(R), \gamma_y(t))>4N(3,0)$ for all $t \in [0, \infty)$. For each natural number $n>R$, let $\gamma_n$ be the geodesic between $\gamma_x(n)$ and $\gamma_y(n)$. 

By Lemma \ref{lem:deltaslim}, we know that the triangle $\gamma_x([0,n]) \cup \gamma_n \cup \gamma_y([0,n])$ is $4N(3,0)$-slim.  From the choice of $R$, it follows that each $\gamma_n$ must intersect a compact ball of radius $4N(3,0)$ centered at $\gamma_x(R)$. By Arzel\'a--Ascoli, there is a subsequence of $\{\gamma_n\}$ which converges to a biinfinite geodesic $\gamma$. Since every $\gamma_n$ is in the $4N(3,0)$ neighborhood of $\gamma_x \cup \gamma_y$, it follows that $\gamma$ must be as well.  From this, a standard argument implies that $\gamma((-\infty, 0]),\gamma([0,\infty))$ are a bounded Hausdorff distance from $\gamma_x, \gamma_y$, respectively, where that bound depends on $N$. 

Finally, it follows from the moreover statement of Lemma \ref{lem:deltaslim} that $\gamma$ is $N'$-Morse, where $N'$ depends on $N$ because each $\gamma_n$ is $N'$-Morse.
\end{proof}

\begin{definition}[Limit geodesics and triangles]
Given $(x_n), (y_n) \in X^{(N)}_e$ and $\lambda^-, \lambda^+ \in \partial X^{(N)}_e$ as above, we call $\gamma_x, \gamma_y$ as described above a \emph{limit legs}, and $\gamma$ a \emph{limit geodesic} for $\lambda^-, \lambda^+$.  We call the triangle formed by $\gamma_x \cup \gamma_y \cup \gamma$ a \emph{limit triangle} based at $e$.
\end{definition}

The next goal is to prove the following two propositions:

\begin{proposition}[Limit triangles are slim]\label{prop:limit triangles are thin}
For any Morse gauge $N$, if $(x_n), (y_n) \subset X^{(N)}_e$ are asymptotic to $\lambda^- \neq \lambda^+ \in \partial X^{(N)}_e$, then any limit triangle is $4N(3,0)$-slim.
\end{proposition}

\begin{proposition}[Limit geodesics are asymptotic]\label{prop:limit geodesics are asymptotic}
For any Morse gauge $N$, there exists a constant $K'>0$ such that if $\gamma, \gamma'$ are limit geodesics with the same endpoints $\lambda^+, \lambda^- \in \partial X^{(N)}_e$, then 
$$d_{Haus}(\gamma, \gamma')<K'.$$
\end{proposition}

Proposition \ref{prop:limit geodesics are asymptotic} is an immediate consequence of Proposition \ref{prop:limit triangles are thin} and following lemma, the proof of which appears in the proof of Theorem 2.8 of \cite{cordes2016stability}:

\begin{lemma}[Limit geodesic rays are asymptotic]\label{lem:limit rays are asymptotic}
For any Morse gauge $N$ and $\lambda\in \partial X^{(N)}_e$, if $\gamma$ and $\gamma'$ are $N$-Morse geodesic rays with endpoint $\lambda$, then $d_{Haus}(\gamma, \gamma')<14N(3,0)$.
\end{lemma}

\begin{proof}[Proof of Proposition \ref{prop:limit triangles are thin}]
By Lemma \ref{lem:limit geodesics are Morse}, $\gamma \subset  \mathcal N_{4N(3,0)}(\gamma_x \cup \gamma_y)$.  Let $w \in \gamma_y$ and assume that $w$ is not within $4N(3,0)$ of a point in $\gamma_x$.  It suffices to find a point on $\gamma$ within $4N(3,0)$ of $w$.  (A similar argument works when $y$ is not within $4N(3,0)$ of $\gamma_y$.)

Let $z \in \gamma$ be the closest point on $\gamma$ to $w$.  By Lemma \ref{lem:deltaslim}, each $\gamma_x([0,n]) \cup \gamma_n \cup \gamma_y[(0,n)]$ is $4N(3,0)$-slim.  Since the $\gamma_n$ subsequentially converges on compact sets to $\gamma$, it follows that for all $\epsilon \geq 0$, we have $d(w, z)\leq4N(3,0)+ \epsilon$. Thus $d(w, z)\leq4N(3,0)$ and the proposition follows from a symmetric argument with $w \in \gamma_x$. 
%
%
%
\end{proof}

\subsection{Asymptoticity}

In this subsection, we study the behavior of geodesics and geodesic rays asymptotic to points in $\partial X^{(N)}_e$.

\begin{definition}[Asymptotic rays]
Let $\lambda \in \partial X^{(N)}_e$ and let $\gamma:[0, \infty) \rightarrow X$ with $\gamma(0)=e$ be a limit leg for $\lambda$ based at $e$.  We say that a geodesic $\gamma':[0, \infty) \rightarrow X$ with $\gamma'(0)=e$ is \emph{asymptotic} to $\lambda$ if there exists $K>0$ such that
$$d_{Haus}(\gamma, \gamma') < K.$$
\end{definition}
 
The following is an immediate consequence of Corollary 2.6 in \cite{Cordes15}:

\begin{lemma}\label{lem:stayclose}
There exists $K_0>0$ and Morse gauge $N'$ depending only on $N$ such that the following holds: For any $\lambda \in \partial X^{(N)}_e$, if $\gamma, \gamma'\colon [0,\infty)\rightarrow X$ are geodesic rays with $\gamma(0)=\gamma'(0)$ which are asymptotic to $\lambda$, then $\gamma, \gamma'$ are $N'$-Morse and
$$d_{Haus}(\gamma,\gamma')<K_0.$$
\end{lemma}

\begin{definition}[Asymptotic, bi-asymptotic]
Let $\gamma:(-\infty, \infty)\rightarrow X$ be a biinfinite geodesic in $X$ with $\gamma(0)$ a closest point to $e$ along $\gamma$.  Let $\lambda \in \partial X^{(N)}_e$.  We say $\gamma$ is \emph{forward asymptotic} to $\lambda$ if for any $N$-Morse geodesic ray $\gamma_{\lambda}:[0, \infty)\rightarrow X$ with $\gamma_{\lambda}(0)=e$, there exists $K>0$ such that
$$d_{Haus}(\gamma([0, \infty), \gamma_{\lambda}([0, \infty))<K.$$
We define \emph{backwards asymptotic} similarly.  If $\gamma$ is forwards, backwards asymptotic to $\lambda, \lambda'$, respectively, then we say $\gamma$ is \emph{bi-asymptotic} to $(\lambda, \lambda')$.
\end{definition}

We note that it is an immediate consequence of Proposition \ref{prop:limit triangles are thin} and Lemma \ref{lem:stayclose} that limit geodesics are bi-asymptotic to their endpoints with a uniform asymptoticity constant.

\begin{lemma} \label{lem:third edge close}
There exists a $K_1>0$ depending only on $N$ such that the following holds: Let $\gamma_+, \gamma_-$ be $N$-Morse geodesic rays with $\gamma_+(0), \gamma_-(0)=e$, and let $\gamma$ be a bi-infinite geodesic such that both $$d_{\mathrm{Haus}}(\gamma_-([0, \infty)) , \gamma((-\infty, 0]))<K \text{ and } d_{\mathrm{Haus}}(\gamma_+([0, \infty)) , \gamma([0, \infty)))< K$$ for some constant $K>0$. Then the triangle $\gamma_- \cup \gamma \cup \gamma_+$ is $K_1$-slim. Furthermore,  there exist $S, R \in [0, \infty)$ such that the following holds:
$$d_{\mathrm{Haus}}(\gamma_-([R, \infty)), \gamma((-\infty, 0])), d_{\mathrm{Haus}}(\gamma_+([S, \infty)) , \gamma([0, \infty))),d(\gamma_-(R), \gamma(0)), d(\gamma_+(S), \gamma(0))< K_1.$$
\end{lemma}

\begin{figure}
  \centering
  \def\svgwidth{3 in}
  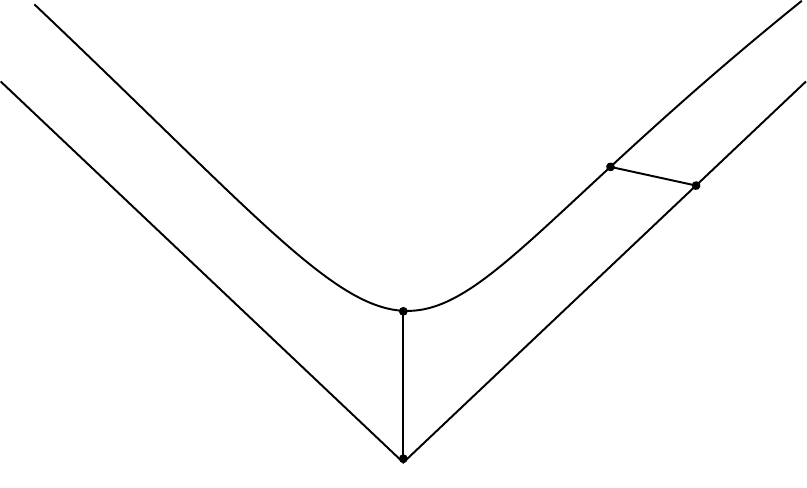
  \caption{Lemma \ref{lem:third edge close}}
\end{figure}

\begin{proof}
Up to reparameterization we may assume that $\gamma(0)$ is the point on $\gamma$ closest to $e$. By assumption we know that there is a constant $K>0$ such that $d_{\mathrm{Haus}}(\gamma([0, \infty)), \gamma_+)<K$. Let $t \in [6K, \infty)$ and let $t' \in [0, \infty)$ be such that $\gamma(t')$ is the closest point along $\gamma$ to $\gamma_+(t)$.  We claim that $\phi= [e, \gamma(0)] \cup \gamma([0,t']) \cup [\gamma(t'), \gamma_+(t)]$ is a $(5,0)$-quasigeodesic.

It suffices to check the standard inequality with vertices in different segments of $\phi$.  First suppose that $u \in [e, \gamma(0)]$ and $v \in [\gamma(0), \gamma(t')]$; the case that $u \in  [\gamma(t'), \gamma_+(t)]$ and $v \in [\gamma(0), \gamma(t)]$ is similar.

We know that $d(u, \gamma(0)) \leq d(u,v)$ because $\gamma(0)$ is a nearest point to $e$ along $\gamma$.  Note that $d(\gamma(0), v) \leq d(u,v) + d(u, \gamma(0))$ by the triangle inequality.  Let $d_\phi(u,v)$ denote the distance along $\phi$ between $u$ and $v$.  We have:
\begin{align*} 
\begin{split}
d(u,v) \leq d_\phi(u,v) =& d(u, \gamma(0)) + d(\gamma(0), v) \\
				\leq& d(u,v) + (d(u,v)+d(u, \gamma(0)) \\
				\leq& 3d(u,v).
\end{split}
\end{align*} 

Now assume that $u \in [e, \gamma(0)]$ and $v \in  [\gamma(t'), \gamma_+(t)]$. 

Let $\xi= [u, \gamma(0)]\cup \gamma([0,t']) \cup [\gamma(t'), v]$ and denote its arclength by $\|\xi\|$. Since $t \in [6K, \infty)$, we know that $d(u,v)\geq t-2K \geq \frac{2}{3}t $ and $d(\gamma(0), \gamma(t')) = t' <t+2K <2t$.
Putting these inequalities together:
\begin{align*}
d(u,v) \leq \|\xi\| =& d(u, \gamma(0)) + d(\gamma(0), \gamma(t')) + d(\gamma(s), v) \\
				<& t/6 + 2t + t/6 < 3t \leq  \frac{9}{2}d(u,v)				 
\end{align*}
Thus $\phi$ is a $(5,0)$-quasi geodesic.

By Lemma \ref{lem:hausdorff}, it follows that $d_{\mathrm{Haus}}(\phi, [e, \gamma_+(t))<N(5,0)$.  Hence there exist $S\geq 0$ and $t_0 > t' - K>0$ such that $d(\gamma(0), \gamma_+(S)), d(\gamma(t'), \gamma_+(t_0))<N(5,0)$.  Let $[\gamma_+(S), \gamma(0)], [\gamma(t'), \gamma_+(t_0)]$ be geodesics and $\phi' = [\gamma_+(S), \gamma(0)] \cup [\gamma(0,t')] \cup [\gamma(t'), \gamma_+(t_0)]$.  Then $\phi'$ is a $(1,2N(5,0))$-quasigeodesic and Lemma \ref{lem:hausdorff} implies that $d_{Haus}(\phi', [\gamma_+(S,t_0)])<N(1,2N(5,0))$ and hence 
$$d_{Haus}([\gamma(0,t')], \gamma_+(S,t_0))<2N(1,2N(5,0))+2N(5,0).$$

Since $t_0 > t' - K >0$ depends only on $t'$ and $N$, we have that $t_0 \rightarrow \infty$ as $t' \rightarrow \infty$, and hence $d_{Haus}(\gamma(0, \infty), \gamma_+(S, \infty))<2N(1,2N(5,0))+2N(5,0)$, as required.  A similar argument with $\gamma_-$ provides $R\geq 0$ such that $d_{Haus}(\gamma(0, -\infty), \gamma_-(R, \infty))<2N(1,2N(5,0))+2N(5,0)$.  

We know by construction that $d(\gamma_-(R), \gamma(0)), d(\gamma_+(S), \gamma(0)) < N(5,0)$, so $d(\gamma_-(R), \gamma_+(S))<2N(5,0)$. We also know by Lemma \ref{lem:deltaslim} that the triangle $\gamma_-([0, R]) \cup [\gamma_-(R), \gamma_+(S)] \cup \gamma_+([0, S])$ is $4N(3,0)$-slim.  Thus if we set $K_1= \max\{(4N(3,0)+2N(5,0)), (2N(1, 2N(5,0))+N(5,0))\}$, then the triangle $\gamma_- \cup \gamma \cup \gamma_+$ is $K_1$-slim.  This completes the proof.
\end{proof}


The following proposition will allow us to define the weak hull in the next subsection:

\begin{proposition} \label{prop:hull geodesics unif close}
There exists $K_2>0$ depending only on $N$ such that the following holds: Let $\lambda_-, \lambda_+ \in \partial X^{(N)}_e$ be distinct points.  If $\gamma, \gamma':(-\infty, \infty)\rightarrow X$ are geodesics with $\gamma(0), \gamma'(0)$ closest points to $e$ along $\gamma, \gamma'$, respectively, such that $\gamma, \gamma'$ are bi-asymptotic to $(\lambda_-, \lambda_+)$, then
$$d_{Haus}(\gamma, \gamma') < K_2.$$
\end{proposition}

\begin{proof} First assume $\gamma$ is a limit geodesic.

Since $\gamma$ is a limit geodesic, Lemma \ref{lem:deltaslim} implies that $\gamma$ is $N'$-Morse where $N'$ depends only on $N$.  By Lemma \ref{lem:third edge close}, we know that there exist $S,R \in [0, \infty)$ such that each of the following holds:

$$d_{\mathrm{Haus}}(\gamma([R, \infty)) , \gamma'((-\infty, 0]))<K_1,  d_{\mathrm{Haus}}(\gamma([S, \infty)) , \gamma'([0, \infty)))< K_1',$$ and
$$d(\gamma(R), \gamma'(0)), d(\gamma(S), \gamma'(0))< K_1',$$
where $K_1'$ depends only on $N'$.

We know then that $d(\gamma(R), \gamma(S))< 2K_1'$. Putting together these facts, we get that the Hausdorff distance between $\gamma, \gamma'$ is less than $2K_1'$.

It follows that any two geodesics $\gamma, \gamma'$ have Hausdorff distance bounded by $K_2=4K_1'$.
\end{proof}

The following is an immediate corollary of Lemma \ref{lem:third edge close} and Proposition \ref{prop:hull geodesics unif close}:

\begin{corollary}
There exists $K_1$ such that for any distinct $\lambda_+, \lambda_- \in \partial X^{(N)}_e$, any geodesic rays $\gamma_+, \gamma_-$ asymptotic to $\lambda_+, \lambda_-$ respectively with $\gamma_+(0) = \gamma_-(0) =e$, and any bi-infinite geodesic $\gamma$ bi-asymptotic to $\lambda_+, \lambda_-$, the triangle $\gamma_+ \cup \gamma \cup \gamma_-$ is $K_1$-slim.
\end{corollary}

\subsection{The weak hull $\mathfrak H_e(G)$}\label{subsec:weak hull} 

We are ready to define the weak hull of $G$ in $X$.

\begin{definition}
The \emph{weak hull} of $G$ in $X$ based at $e \in X$, denoted $\mathfrak{H}_e(G)$, is the collection 
 of all biinfinite rays $\gamma$ which are bi-asymptotic to $(\lambda, \lambda')$ for some $\lambda \neq \lambda' \in \Lambda_e(G)$.
\end{definition}

Note that Proposition \ref{prop:hull geodesics unif close} says that any two rays $\gamma, \gamma' \in \mathfrak{H}_e(G)$ that are both bi-asymptotic to $(\lambda, \lambda')$ have bounded Hausdorff distance, where that bound depends on the stratum $X^{(N)}_e$ in which the sequences defining $\lambda, \lambda'$ live. Also note that the proof of Lemma \ref{lem:limit geodesics are Morse} implies that limit geodesics are in $\mathfrak{H}_e(G)$.

The following lemma is evident from the definitions:

\begin{lemma} \label{lem:G-invariance of weak hull}
If $|\Lambda_e(G)| \geq 2$, then $\mathfrak H_e(G)$ is nonempty and $G$-invariant.
\end{lemma}

The following is an interesting question:

\begin{question}
If $|\Lambda(G)|\neq \emptyset$, then must we have in fact $|\Lambda(G)|\geq 2$?
\end{question}

In the case of $\mathrm{CAT}(0)$ spaces, this question has been affirmatively answered \cite[Lemma 4.9]{murray2015topology}.

\section{Boundary convex cocompactness and stability}

In this section, we prove the main theorem, Theorem \ref{thm:cc and stab}, namely that boundary convex cocompactness as in Definition \ref{defn:cc subspace} and stability as in Definition \ref{def:stability} are equivalent.

\subsection{Compact of limit sets have stable weak hulls}

For the rest of this section, fix a group $G$ acting by isometries on a proper geodesic metric space $X$.

In this subsection, we will prove that if $G$ has a compact limit set, then $\mathfrak H_e(G)$ is stable.

\begin{lemma}\label{lem:compact uniform gauge}
If $K \subset \partial_M X$ is nonempty and compact, then for any $e\in X$ there exists $N>0$ such that $K \subset \partial X^{(N)}_e$.
\end{lemma}

\begin{proof}
We will closely follow the proof of Lemma 3.3 in \cite{murray2015topology}.

Since X is proper, we use the fact that $\partial X^{(N)}_e$ is homeomorphic to $\partial_M^N X$ \cite[Theorem 3.14]{cordes2016stability}.

Assume that $K$ is not contained in $\partial_M^N X$ for any Morse gauge $N$. Then we know that there is a sequence  $(\alpha_i) \subset K$ of $N_i$-Morse geodesics where $N_i > N_{i-1} +1$ for all $i$.  Let $A_n=\{\alpha_i\}_{i \geq n+1}$. We note that for all $n$ and $N$, $A_n \cap \partial_M^N X$ is finite. Since singletons are closed in each $\partial_M^N X$, each of which is Hausdorff, it follows that singletons are closed in $\partial_M X$. Thus $A_n$ is closed in $\partial_M X$ by the definition of the direct limit topology.

The collection $\{\partial_M X \char`\\ A_n\}_{n \in \N}$ is an open cover of $K$, but each $\partial_M X \setminus A_n$ only contains a finite number of the $\alpha_i$, so that any finite subcollection of $\{\partial_M X \char`\\ A_n\}_{n \in \N}$ will only contain finitely many $\alpha_i$.  This contradicts the fact that $K$ is compact, completing the proof.
\end{proof}

The following proposition is the main technical statement of this section:

\begin{figure}
  \centering
  \def\svgwidth{3 in}
  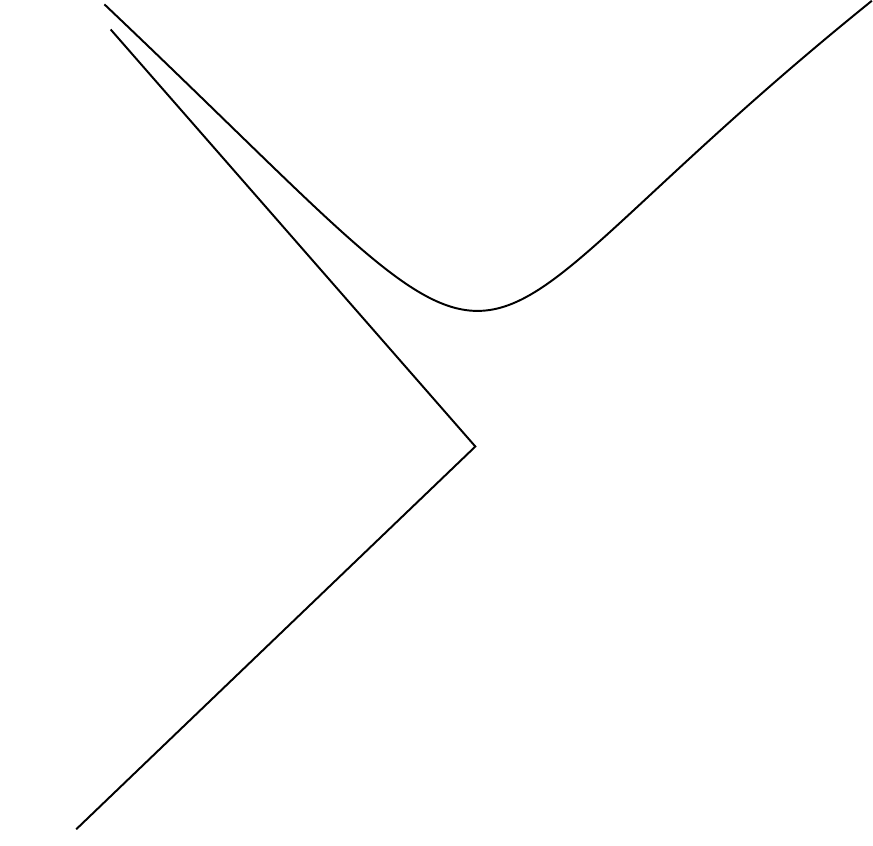
  \caption{Proposition \ref{prop:compact hyp hull}}
\end{figure}

\begin{proposition}[Compact limit sets have stable hulls] \label{prop:compact hyp hull}
If $\Lambda_e(G) \subset \partial X^{(N)}_e$ is compact for some (any) $e \in X$, then for each $e \in X$ there exists a Morse gauge $N'$ such that $\mathfrak H_e(G)$ is $N'$-stable.
\end{proposition}

\begin{proof}
Let $x,y \in \mathfrak H_e(G)$.  By Lemma \ref{lem:limit geodesics are Morse}, we may assume that $x,y$ do not lie on the same geodesic.  Let $[x,y]$ be any geodesic between $x$ and $y$.

Since $\mathfrak H_e(G)$ is a subspace with the induced metric, it suffices to prove that $[x,y]$ is uniformly Morse.  Moreover, since every hull geodesic lies within uniform Hausdorff distance of a limit geodesic by Proposition \ref{prop:hull geodesics unif close} and Lemma \ref{lem:compact uniform gauge}, it suffices to consider the case when $x$ and $y$ lie on limit geodesics by Lemma \ref{lem:hausdorff}.


Let $\gamma_x, \gamma_y$ be distinct limit geodesics on which $x,y$ lie, respectively.  By Proposition \ref{prop:limit triangles are thin}, there exist limit legs $\gamma'_x, \gamma'_y$ based at $e$ and points $x' \in \gamma'_x, y' \in \gamma'_y$ such that $d(x,x'), d(y,y')\leq 4N(3,0)$.

Let $\gamma'$ be the limit geodesic which forms a limit triangle with limit legs $\gamma'_x$ and $\gamma'_y$. Note by definition $\gamma' \in  \mathfrak H_e(G)$. By Proposition \ref{prop:limit triangles are thin}, the limit triangle $\gamma' \cup \gamma'_x \cup \gamma'_y$ is $4N(3,0)$-thin, so there exist $x'', y'' \in \gamma'$ with $d(x',x''), d(y',y'') < 4N(3,0)$, and hence $d(x,x''), d(y,y'') < 8N(3,0)$.

Let $[x'',x], [y,y'']$ be any geodesics between $x,x''$ and $y,y''$, respectively.  Then the concatenation $\sigma= [x'',x] \cup [x,y] \cup [y,y'']$ gives a $(1, 16N(3,0) )$-quasigeodesic with endpoints $x'',y''$ on the $N'$-Morse geodesic $\gamma'$.  If $[x'',y''] \subset \gamma'$ is the subsegment of $\gamma'$ between $x''$ and  $y''$, then $\sigma$ is in the $N'(1, 16N(3,0))$ neighborhood of $[x'',y'']$.  An easy argument then implies that $\sigma$ is uniformly Morse, from which it follows immediately that $[x,y]$ is uniformly Morse. \end{proof}

\subsection{Boundary convex cocompactness implies stability}

We now prove the first direction of the main theorem:

\begin{theorem} \label{thm:cocompact implies stability}
Suppose $G$ acts by isometries on a proper geodesic metric space $X$ such that 
\begin{enumerate}
\item The action of $G$ on $X$ is proper,
\item $\Lambda_e(G)$ is nonempty and compact for any $e \in X$, and
\item $G$ acts cocompactly on $\mathfrak H_e(G)$ for any $e \in X$.
\end{enumerate}
Then any orbit of $G$ is a stable subspace of $X$ and the orbit map extends continuously and $G$-equivariantly to an embedding of $\partial_{Gr} G$ into $\partial_M X_e$ which is a homeomorphism onto its image $\Lambda_e(G)$.
\end{theorem}

\begin{proof}

Since $G$ acts properly and cocompactly on $\mathfrak H_e(G)$, it follows that the orbit map $g \mapsto g \cdot e$ is a quasiisometry $G \rightarrow \mathfrak H_e(G)$, where we consider $\mathfrak H_e(G)$ with the metric induced from $X$.  In particular, $G \cdot e$ is quasiisometrically embedded in $X$.

Let $x \in \mathfrak H_e(G)$ and consider $G \cdot x \subset \mathfrak H_e(G) \subset X$.  Let $y, z \in G \cdot x$.  By Proposition \ref{prop:compact hyp hull}, there exists a Morse gauge $N$ such that any geodesic $[y,z]$ between $y,z$ is $N$-Morse.  Hence, any $(K,C)$-quasigeodesic between $y,z$ must stay within $N(K,C)$ of $[y,z]$, implying that $G\cdot x$ is a stable subspace of $X$.  In particular, $G$ is hyperbolic, so $\partial_{Gr} G$ is defined.

By Corollary \ref{cor:stable morse preserving}, the orbit map $g \mapsto g \cdot e$ induces a topological embedding $f: \partial_{Gr}G \rightarrow \partial_M X$ which gives a homeomorphism $f:\partial_{Gr} G \rightarrow \Lambda_e(G)$.  This completes the proof. \end{proof}

\subsection{Stability implies boundary convex cocompactness}


Finally, we prove the second direction:

\begin{theorem} \label{thm:stability implies cocompact}
Suppose $G$ acts by isometries on a proper geodesic metric space $X$ such that any orbit map of $G$ is an infinite diameter stable subspace of $X$.  Then:
\begin{enumerate}
\item \label{limit set compact} $\Lambda(G)$ is nonempty and compact,
\item $\mathfrak H_e(G)$ is a stable subset of $X$,
\item $G$ acts cocompactly on $\mathfrak H_e(G)$, and
\item Any orbit map extends continuously and $G$-equivariantly to an embedding of $\partial_{Gr} G$ into $\partial_M X_e$ which is a homeomorphism onto its image $\Lambda_e(G)$.
\end{enumerate}
\end{theorem}

\begin{proof}
$G$ is hyperbolic because its orbit is stable, so we know that $\partial_M G = \partial_{Gr} G$ is compact.  By Corollary \ref{cor:stable morse preserving}, any orbit map $g \mapsto g \cdot e$ induces a homeomorphism $\partial_{Gr} G \rightarrow \Lambda_e(G)$, giving us (1) and (4).

Thus by Proposition \ref{prop:compact hyp hull}, the weak hull $\mathfrak H_e(G)$ is stable in $X$, proving (3).  It remains to prove that $G$ acts cocompactly on $\mathfrak H_e(G)$.

To see this, let $z \in \mathfrak H_e(G)$.  By Proposition \ref{prop:hull geodesics unif close}, we may assume that $z \in \gamma$ for some limit geodesic $\gamma$ with ends $(g_n \cdot e), (h_n \cdot e)$.  Let $\gamma_g, \gamma_h$ be the limit legs of the corresponding limit triangle for $\gamma$.  By Lemma \ref{prop:limit triangles are thin}, there exists $x \in \gamma_g \cup \gamma_h$ such that $d(x,z) < 4N(3,0)$, where $N$ is the Morse gauge such that $\Lambda_e(G) \subset \partial X^{(N)}_e$.

Without loss of generality, assume that $x \in \gamma_g$.  By definition of $\gamma_g$, there exists a sequence of $N$-Morse geodesics $\gamma_n = [e, g_n \cdot e]$ which subsequentially converges to $\gamma_g$ uniformly on compact sets.  Hence, by passing to a subsequence if necessary, for any $\epsilon>0$ and every $T>0$, there exists $M>0$ such that for any $n\geq M$, we have
$$d_{\text{Haus}}(\gamma_n([0,T]), \gamma_g([0,T]))<\epsilon.$$

Taking $T>0$ sufficiently large so that $x \in \gamma_g([0,T])$ and taking $n>M$ for $M$ corresponding to this $T$, we may assume there exists $y \in \gamma_n$ such that $d(y,x) < \epsilon$.

Since $G\cdot e$ is quasiisometrically embedded, there exist $K,C$ such that any geodesic in $G$ between $1$ and $g_n$ maps to a $(K,C)$-quasigeodesic $q_n$ between $e$ and $g_n \cdot e$.  Let $q'_n$ denote the $(K,C +(K+C))$-quasigeodesic obtained by connecting successive vertices of $q_n$ by geodesics of length at most $K+C$.  Since $[e,g_n \cdot e]$ is $N$-Morse, it follows that 


there exists $w' \in q'_n$ such that $d(w',y) < 2N(K,2C+K)$, and thus $w \in q_n$ with $d(w,y)<2N(K,2C+K) + \frac{K+C}{2}$.  It follows that
$$d(z,w) \leq d(z,x) + d(x, y) + d(y,w) < 4N(3,0) + \epsilon + \left(2N(K,2C+K)+\frac{K+C}{2}\right)$$
which is a constant depending only on $N$ and the quasiisometric embedding constants of $G$ in $X$.  Hence $G$ acts cocompactly on $\mathfrak H_e(G)$, as required.
\end{proof}

\bibliographystyle{alpha}
\bibliography{boundarycc}
\end{document}

%% file: lemma312.pdf_tex
\begingroup%
  \makeatletter%
  \providecommand\color[2][]{%
    \errmessage{(Inkscape) Color is used for the text in Inkscape, but the package 'color.sty' is not loaded}%
    \renewcommand\color[2][]{}%
  }%
  \providecommand\transparent[1]{%
    \errmessage{(Inkscape) Transparency is used (non-zero) for the text in Inkscape, but the package 'transparent.sty' is not loaded}%
    \renewcommand\transparent[1]{}%
  }%
  \providecommand\rotatebox[2]{#2}%
  \ifx\svgwidth\undefined%
    \setlength{\unitlength}{232.31593921bp}%
    \ifx\svgscale\undefined%
      \relax%
    \else%
      \setlength{\unitlength}{\unitlength * \real{\svgscale}}%
    \fi%
  \else%
    \setlength{\unitlength}{\svgwidth}%
  \fi%
  \global\let\svgwidth\undefined%
  \global\let\svgscale\undefined%
  \makeatother%
  \begin{picture}(1,0.60396801)%
    \put(0,0){\includegraphics[width=\unitlength,page=1]{lemma312.pdf}}%
    \put(0.71789512,0.19004124){\color[rgb]{0,0,0}\makebox(0,0)[lb]{\smash{$\gamma_+$}}}%
    \put(0.17127723,0.27408396){\color[rgb]{0,0,0}\makebox(0,0)[lb]{\smash{$\gamma_-$}}}%
    \put(0.19075712,0.46157761){\color[rgb]{0,0,0}\makebox(0,0)[lb]{\smash{$\gamma$}}}%
    \put(0.48529457,0.25518753){\color[rgb]{0,0,0}\makebox(0,0)[lb]{\smash{$\gamma(0)$}}}%
    \put(0.67704299,0.42415971){\color[rgb]{0,0,0}\makebox(0,0)[lb]{\smash{$\gamma(t')$}}}%
    \put(0.88834172,0.354161){\color[rgb]{0,0,0}\makebox(0,0)[lb]{\smash{$\gamma_+(t)$}}}%
    \put(0.48643931,0.00316321){\color[rgb]{0,0,0}\makebox(0,0)[lb]{\smash{$e$}}}%
    \put(0,0){\includegraphics[width=\unitlength,page=2]{lemma312.pdf}}%
    \put(0.61784952,0.10518933){\color[rgb]{0,0,0}\makebox(0,0)[lb]{\smash{$\gamma_+(S)$}}}%
    \put(0,0){\includegraphics[width=\unitlength,page=3]{lemma312.pdf}}%
    \put(0.83479,0.28681056){\color[rgb]{0,0,0}\makebox(0,0)[lb]{\smash{$\gamma_+(t')$}}}%
  \end{picture}%
\endgroup%

%% file: prop43.pdf_tex
\begingroup%
  \makeatletter%
  \providecommand\color[2][]{%
    \errmessage{(Inkscape) Color is used for the text in Inkscape, but the package 'color.sty' is not loaded}%
    \renewcommand\color[2][]{}%
  }%
  \providecommand\transparent[1]{%
    \errmessage{(Inkscape) Transparency is used (non-zero) for the text in Inkscape, but the package 'transparent.sty' is not loaded}%
    \renewcommand\transparent[1]{}%
  }%
  \providecommand\rotatebox[2]{#2}%
  \ifx\svgwidth\undefined%
    \setlength{\unitlength}{251.28142514bp}%
    \ifx\svgscale\undefined%
      \relax%
    \else%
      \setlength{\unitlength}{\unitlength * \real{\svgscale}}%
    \fi%
  \else%
    \setlength{\unitlength}{\svgwidth}%
  \fi%
  \global\let\svgwidth\undefined%
  \global\let\svgscale\undefined%
  \makeatother%
  \begin{picture}(1,0.96887299)%
    \put(0,0){\includegraphics[width=\unitlength,page=1]{prop43.pdf}}%
    \put(0.77528217,0.21417194){\color[rgb]{0,0,0}\makebox(0,0)[lb]{\smash{$\gamma_y$}}}%
    \put(0.74473639,0.81889652){\color[rgb]{0,0,0}\makebox(0,0)[lb]{\smash{$\gamma_x$}}}%
    \put(0,0){\includegraphics[width=\unitlength,page=2]{prop43.pdf}}%
    \put(0.32246732,0.77970133){\color[rgb]{0,0,0}\makebox(0,0)[lb]{\smash{$x$}}}%
    \put(0.42444554,0.26688009){\color[rgb]{0,0,0}\makebox(0,0)[lb]{\smash{$y$}}}%
    \put(0,0){\includegraphics[width=\unitlength,page=3]{prop43.pdf}}%
    \put(0.35569622,0.32990043){\color[rgb]{0,0,0}\makebox(0,0)[lb]{\smash{$y'$}}}%
    \put(0.2560096,0.68739731){\color[rgb]{0,0,0}\makebox(0,0)[lb]{\smash{$x'$}}}%
    \put(0,0){\includegraphics[width=\unitlength,page=4]{prop43.pdf}}%
    \put(0.17641108,0.65201674){\color[rgb]{0,0,0}\makebox(0,0)[lb]{\smash{$x''$}}}%
    \put(0.23697873,0.34339765){\color[rgb]{0,0,0}\makebox(0,0)[lb]{\smash{$y''$}}}%
    \put(0,0){\includegraphics[width=\unitlength,page=5]{prop43.pdf}}%
    \put(0.10527318,0.86653146){\color[rgb]{0,0,0}\makebox(0,0)[lb]{\smash{$\gamma'_x$}}}%
    \put(0.07961689,0.0748987){\color[rgb]{0,0,0}\makebox(0,0)[lb]{\smash{$\gamma'_y$}}}%
    \put(0.0623117,0.21103917){\color[rgb]{0,0,0}\makebox(0,0)[lb]{\smash{$\gamma'$}}}%
    \put(0,0){\includegraphics[width=\unitlength,page=6]{prop43.pdf}}%
    \put(0.55586907,0.45445536){\color[rgb]{0,0,0}\makebox(0,0)[lb]{\smash{$e$}}}%
  \end{picture}%
\endgroup%